\documentclass[a4paper,twoside]{article}
\usepackage{a4}
\usepackage{amssymb}
\usepackage{amsmath}
\usepackage{upref}
\usepackage[pagebackref,colorlinks,citecolor=blue,linkcolor=blue,urlcolor=blue]{hyperref}
\usepackage[active]{srcltx}
\usepackage[dvipsnames]{color}
\allowdisplaybreaks[2] 
\newcount\minutes \newcount\hours
\hours=\time \divide\hours 60 \minutes=\hours
 \multiply\minutes-60
\advance\minutes \time
\newcommand{\klockan}{\the\hours:{\ifnum\minutes<10 0\fi}\the\minutes}
\newcommand{\tid}{\today\ \klockan}
\newcommand{\prtid}{\smash{\raise 10mm \hbox{\LaTeX ed \tid}}}
\renewcommand{\prtid}{}
\makeatletter  \headheight 10pt
\def\sectionmark#1{} 
\def\subsectionmark#1{}
\newcommand{\sectnr}{\ifnum \c@secnumdepth >\z@
                 \thesection.\hskip 1em\relax \fi}
\def\@evenhead{\footnotesize\rm\thepage\hfil\leftmark\hfil\llap{\prtid}}
\def\@oddhead{\footnotesize\rm\rlap{\prtid}\hfil\rightmark\hfil\thepage}
\def\tableofcontents{\section*{Contents} 
 \@starttoc{toc}}
\makeatother
\makeatletter
\def\@biblabel#1{#1.}
\makeatother
\makeatletter
\let\Thebibliography=\thebibliography
\renewcommand{\thebibliography}[1]{\def\@mkboth##1##2{}\Thebibliography{#1}
\addcontentsline{toc}{section}{References}
\frenchspacing 
\setlength{\@topsep}{0pt}
\setlength{\itemsep}{0pt}%
\setlength{\parskip}{0pt plus 2pt}%
}
\makeatother
\makeatletter
\def\mdots@{\mathinner.\nonscript\!.%
 \ifx\next,.\else\ifx\next;.\else\ifx\next..\else
 \nonscript\!\mathinner.\fi\fi\fi}
\let\ldots\mdots@
\let\cdots\mdots@
\let\dotso\mdots@
\let\dotsb\mdots@
\let\dotsm\mdots@
\let\dotsc\mdots@
\def\vdots{\vbox{\baselineskip2.8\p@ \lineskiplimit\z@
    \kern6\p@\hbox{.}\hbox{.}\hbox{.}\kern3\p@}}
\def\ddots{\mathinner{\mkern1mu\raise8.6\p@\vbox{\kern7\p@\hbox{.}}%
    \raise5.8\p@\hbox{.}\raise3\p@\hbox{.}\mkern1mu}}
\makeatother
\makeatletter
\let\Enumerate=\enumerate
\renewcommand{\enumerate}{\Enumerate%
\setlength{\@topsep}{0pt}
\setlength{\itemsep}{0pt}%
\setlength{\parskip}{0pt plus 1pt}%
\renewcommand{\theenumi}{\textup{(\alph{enumi})}}%
\renewcommand{\labelenumi}{\theenumi}%
}
\let\endEnumerate=\endenumerate
\renewcommand{\endenumerate}{\endEnumerate\unskip}
\makeatother
\makeatletter
\def\@seccntformat#1{\csname the#1\endcsname.\quad}
\makeatother
\newcommand{\authortitle}[2]{\author{#1}\title{#2}\markboth{#1}{#2}}
\newcommand{\art}[6]{{\sc #1, \rm #2, \it #3 \bf #4 \rm (#5), \mbox{#6}.}}
\newcommand{\auth}[2]{{#1, #2.}}

\newcommand{\arttoappear}[3]{{\sc #1, \rm #2, to appear in \it #3}}
\newcommand{\book}[3]{{\sc #1, \it #2, \rm #3.}}
\newcommand{\AND}{{\rm and }}
\RequirePackage{amsthm}
\newtheoremstyle{descriptive}%
  {\topsep}   
  {\topsep}   
  {\rmfamily} 
  {}          
  {\bfseries} 
  {.}         
  { }         
  {}          
\newtheoremstyle{propositional}%
  {\topsep}   
  {\topsep}   
  {\itshape}  
  {}          
  {\bfseries} 
  {.}         
  { }         
  {}          
\theoremstyle{propositional}
\newtheorem{thm}{Theorem}[section]
\newtheorem{prop}[thm]{Proposition}
\newtheorem{lem}[thm]{Lemma}
\newtheorem{cor}[thm]{Corollary}
\theoremstyle{descriptive}
\newtheorem{deff}[thm]{Definition}
\newtheorem{remark}[thm]{Remark}
\newtheorem{openprob}[thm]{Open problem}
\makeatletter
\renewenvironment{proof}[1][\proofname]{\par
  \pushQED{\qed}%
  \normalfont
  \trivlist
  \item[\hskip\labelsep
        \itshape
    #1\@addpunct{.}]\ignorespaces
}{%
  \popQED\endtrivlist\@endpefalse
} \makeatother
\newcommand{\setm}{\setminus}
\renewcommand{\emptyset}{\varnothing}
\newcommand{\Cp}{{C_p}}
\newcommand{\CpU}{{C_p^U}}
\newcommand{\CpE}{{C_p^E}}
\newcommand{\CpX}{{C_p^X}}
\DeclareMathOperator{\capp}{cap}
\newcommand{\cp}{\capp_p}
\DeclareMathOperator{\diam}{diam} 
\DeclareMathOperator{\Lip}{Lip}
\DeclareMathOperator{\fineint}{fine-int}
\newcommand{\psubset}{\mathrel{\overset{\scriptscriptstyle p}{\subset}}}
\newcommand{\pSubset}{\mathrel{\overset{\scriptscriptstyle p}{\Subset}}}
\newcommand{\bdry}{\partial}
\newcommand{\bdy}{\bdry}
\newcommand{\loc}{_{\rm loc}}
{\catcode`p =12 \catcode`t =12 \gdef\eeaa#1pt{#1}}      
\def\accentadjtext#1{\setbox0\hbox{$#1$}\kern   
                \expandafter\eeaa\the\fontdimen1\textfont1 \ht0 }
\def\accentadjscript#1{\setbox0\hbox{$#1$}\kern 
                \expandafter\eeaa\the\fontdimen1\scriptfont1 \ht0 }
\def\accentadjscriptscript#1{\setbox0\hbox{$#1$}\kern   
                \expandafter\eeaa\the\fontdimen1\scriptscriptfont1 \ht0 }
\def\accentadjtextback#1{\setbox0\hbox{$#1$}\kern       
                -\expandafter\eeaa\the\fontdimen1\textfont1 \ht0 }
\def\accentadjscriptback#1{\setbox0\hbox{$#1$}\kern     
                -\expandafter\eeaa\the\fontdimen1\scriptfont1 \ht0 }
\def\accentadjscriptscriptback#1{\setbox0\hbox{$#1$}\kern 
                -\expandafter\eeaa\the\fontdimen1\scriptscriptfont1 \ht0 }
\def\itoverline#1{{\mathsurround0pt\mathchoice
        {\rlap{$\accentadjtext{\displaystyle #1}
                \accentadjtext{\vrule height1.593pt}
                \overline{\phantom{\displaystyle #1}
                \accentadjtextback{\displaystyle #1}}$}{#1}}
        {\rlap{$\accentadjtext{\textstyle #1}
                \accentadjtext{\vrule height1.593pt}
                \overline{\phantom{\textstyle #1}
                \accentadjtextback{\textstyle #1}}$}{#1}}
        {\rlap{$\accentadjscript{\scriptstyle #1}
                \accentadjscript{\vrule height1.593pt}
                \overline{\phantom{\scriptstyle #1}
                \accentadjscriptback{\scriptstyle #1}}$}{#1}}
        {\rlap{$\accentadjscriptscript{\scriptscriptstyle #1}
                \accentadjscriptscript{\vrule height1.593pt}
                \overline{\phantom{\scriptscriptstyle #1}
                \accentadjscriptscriptback{\scriptscriptstyle #1}}$}{#1}}}}
\newcommand{\alp}{\alpha}
\newcommand{\al}{\alpha}
\newcommand{\be}{\beta}
\newcommand{\eps}{\varepsilon}
\newcommand{\la}{\lambda}
\newcommand{\ga}{\gamma}
\newcommand{\Ga}{\Gamma}
\newcommand{\Om}{\Omega}
\renewcommand{\phi}{\varphi}
\newcommand{\p}{{$p\mspace{1mu}$}}
\newcommand{\clEp}{{\itoverline{E}\mspace{1mu}}^p}
\newcommand{\bdyp}{\bdy_p} 
\newcommand{\R}{\mathbf{R}}
\newcommand{\eR}{{\overline{\R}}}
\newcommand{\K}{{\cal K}}
\newcommand{\limplus}{{\mathchoice{\vcenter{\hbox{$\scriptstyle +$}}}
  {\vcenter{\hbox{$\scriptstyle +$}}}
  {\vcenter{\hbox{$\scriptscriptstyle +$}}}
  {\vcenter{\hbox{$\scriptscriptstyle +$}}}
}}
\newcommand{\Np}{N^{1,p}}
\newcommand{\Nploc}{N^{1,p}\loc}
\newcommand{\Npploc}{N^{1,p}_{\textup{fine-loc}}}
\newcommand{\Lploc}{L^{p}\loc}
\numberwithin{equation}{section}
\newcommand{\imp}{\ensuremath{\Rightarrow}}

\newenvironment{ack}{\medskip{\it Acknowledgement.}}{}

\begin{document}

\authortitle{Anders Bj\"orn, Jana Bj\"orn and Visa Latvala}
            {Convergence and local-to-global results
for \p-superminimizers on
quasiopen sets}
\author{
Anders Bj\"orn \\
\it\small Department of Mathematics, Link\"oping University, SE-581 83 Link\"oping, Sweden\\
\it \small anders.bjorn@liu.se, ORCID\/\textup{:} 0000-0002-9677-8321
\\
\\
Jana Bj\"orn \\
\it\small Department of Mathematics, Link\"oping University, SE-581 83 Link\"oping, Sweden\\
\it \small jana.bjorn@liu.se, ORCID\/\textup{:} 0000-0002-1238-6751
\\
\\
Visa Latvala \\
\it\small Department of Physics and Mathematics,
University of Eastern Finland, \\
\it\small  P.O. Box 111, FI-80101 Joensuu,
Finland\/{\rm ;} \\
\it \small visa.latvala@uef.fi, ORCID\/\textup{:} 0000-0001-9275-7331
}
\date{}

\maketitle

\noindent{\small
{\bf Abstract.}
In this paper,  several convergence results for
fine  \p-(super)minimizers on quasiopen sets
in metric spaces are obtained.
For this purpose, we
deduce a Caccioppoli-type inequality
and  local-to-global principles for fine \p-(super)minimizers
on quasiopen sets.
A substantial part of these considerations is to show that
the functions belong to a suitable local fine Sobolev space.
We prove our results for a complete metric space equipped with a doubling measure
supporting a \p-Poincar\'e inequality 
with $1<p< \infty$. However,
most of the results are new also for unweighted $\R^n$.
}

\medskip


\medskip


\section{Introduction}
\label{sect-intro}

Superharmonic functions in classical potential theory are in general
only continuous with respect to the so-called fine topology, rather than the
standard topology.
This is one reason for studying
finely (super)harmonic functions on
finely open sets.
Their theory  has been developed
mainly in the linear axiomatic setting, 
see e.g.\ the monographs
by Fuglede~\cite{Fug} and~Luke\v{s}--Mal\'y--Zaj\'i\v{c}ek~\cite{LuMaZa}.

Finely open sets are 
special cases of
quasiopen sets, which 
coincide with
the superlevel sets of Sobolev functions.
As such, they carry nontrivial Sobolev test functions and are thus
suitable for partial differential equations, such as the \p-Laplace
equation $\Delta_pu=0$.
In the nonlinear case, i.e.\ for $p\ne2$, a study of 
fine \p-(super)solutions to such (and more general)
equations was conducted in 1992 by Kilpel\"ainen--Mal\'y~\cite{KiMa92}
on quasiopen sets in unweighted $\R^n$.
That theory was further extended by Latvala~\cite{LatPhD}, \cite{Lat00}, in
particular for $p=n$. 
Quasiopen sets also 
appear as solutions to certain shape optimization problems
as e.g.\ in Buttazzo--Dal Maso~\cite{buttazzo-dalMaso}
and Fusco--Mukherjee--Zhang~\cite{FuscoMZ}.

In this paper we continue our study
of fine \p-(super)minimizers on quasiopen sets in metric spaces
initiated in~\cite{BBLat4}.
We consider a complete metric space $X$ equipped with a doubling measure $\mu$
supporting a \p-Poincar\'e inequality with $1<p< \infty$.
In this setting, \p-harmonic functions are defined 
as minimizers of the \p-energy using upper gradients.
Since there is (in general) no differential equation, we talk about
fine \p-(super)minimizers, rather than fine
\p-(super)solutions as in~\cite{KiMa92}. 
On $\R^n$, these two notions coincide~\cite[Proposition~5.3]{BBLat4}.
From now on we drop the parameter $p$ and just write fine (super)minimizer.

It turns out that
metric spaces are well suited for studying these notions on finely and
quasiopen sets, since such sets can easily be seen as metric spaces in
their own right, see \cite{BBnonopen}--\cite{BBMaly}.
The function space naturally associated with fine superminimizers
on metric spaces is the
Newton--Sobolev space $\Np$ and its local fine version $\Npploc$
defined through compactly contained \p-strict subsets.
A set $V \subset U$  is a \emph{\p-strict subset} of $U$ if
there is 
$\eta \in \Np(X)$ such that $\eta =1$ on~$V$ and $\eta=0$
outside $U$.
This is denoted
$V \psubset U$, and similarly $V \pSubset U$
when also $V\Subset U$.

The following is a special case of our main result.
Even this special case is stronger than the earlier monotone convergence
results on unweighted $\R^n$ proved in~\cite[Theorems~4.2 and~4.3]{KiMa92},
see the introduction to Section~\ref{sect-seq}.
We are not aware of any other similar convergence results
in the nonlinear case.

\begin{thm} \label{thm-gen-seq-intro}
Let $U \subset X$ be quasiopen and
$\{u_j\}_{j=1}^\infty$ be a sequence of fine superminimizers in $U$
such that $u_j \ge f$ a.e.\ in $U$, $j=1,2,\dots$\,,
for some $f \in \Npploc(U)$.

Assume that one of the following conditions holds
for $u:=\liminf_{j\to\infty} u_j$\/\textup{:}
\begin{enumerate}
\item \label{i-a} 
$u$ is a.e.-bounded\/\textup{;}
\item \label{i-b}
$u \in \Npploc(U)$\/\textup{;}
\item \label{i-c}
$|u-f_V| \le M_V$ a.e.\ in every finely open $V \psubset U$
for some $f_V\in\Np(V)$ and $M_V \ge 0$.
\end{enumerate}
Then $u$ is a fine superminimizer in $U$.
\end{thm}

For monotone sequences of fine minimizers this reduces to
the following simpler statement.

\begin{cor} \label{cor-monotone-min-intro}
Let $U \subset X$ be quasiopen and
$\{u_j\}_{j=1}^\infty$ be a monotone, or uniformly converging, sequence of fine minimizers in $U$.

Then $u:=\lim_{j\to\infty} u_j$ is a fine minimizer in~$U$,
provided that 
one of the conditions~\ref{i-a}--\ref{i-c} in Theorem~\ref{thm-gen-seq-intro}
holds.
\end{cor}

One of our motivations  for studying such
convergence results is 
that they provide fundamental tools for further studies,
such as fine Perron solutions, which we pursue in a
forthcoming paper.
For \p-Laplace type equations on open sets in weighted $\R^n$, such monotone convergence
results were obtained and systematically used in 
Heinonen--Kilpel\"ainen--Martio~\cite{HeKiMa}.
For developing fine potential theory, convergence results
are even more central since there are less other tools available.

A property usually taken as an axiom in most
axiomatic potential theories is the sheaf property,
which says that if $u$ is superharmonic in $U_\alp$ for each $\alp$,
then it is superharmonic in $\bigcup_{\alp} U_\alp$.
The definitions of 
finely superharmonic functions in Fuglede~\cite[Definition~8.1]{Fug}
(in the linear theory)
and 
in Kilpel\"ainen--Mal\'y~\cite[Definition~5.5]{KiMa92}
(in the nonlinear theory)
are based on (and yield) the sheaf property.
This property
is usually obvious for
supersolutions of differential equations on open sets as well,
but is not known in connection with  minimization problems,
see \cite[Open problems~9.22 and~9.23]{BBbook}.

Our notion of fine superminimizers
is on open sets equivalent to the standard 
superminimizers,
by Corollary~5.6 in~\cite{BBLat4}. 
The same is true for the
fine supersolutions introduced 
in~\cite{KiMa92}.
In contrast, Fuglede~\cite[Th\'eor\`eme~3.2]{Fug74} gave an example
of a finely harmonic function  in the sense of~\cite{Fug}
on the entire space $\R^n$, $n \ge 3$,
which has an essential singularity at the origin and is
therefore
\emph{not} superharmonic on all of $\R^n$.
On the other hand, Luke\v{s}--Mal\'y--Zaj\'i\v{c}ek~\cite[Section~12.A]{LuMaZa}
use a more restrictive definition of finely superharmonic functions than Fuglede.
Their definition lacks the sheaf property, but
instead has the property that  a finely
superharmonic function on an open set is a standard superharmonic function,
similarly to our fine superminimizers.
See also Remark~\ref{rmk-local-to-global-sheaf}.

As a partial substitute for the absent sheaf property, we obtain several
``local-to-global principles'' in Section~\ref{sect-char}.
The following is a special case of Corollary~\ref{cor:localtoglobal-2}, 
which generalizes \cite[Theorem~4.2\,(a)]{KiMa92}.
The removability of $E$ generalizes \cite[Lemma~8.1]{BBLat4}.

\begin{prop}  \label{prop-local-global-intro}
Let $U \subset X$ be quasiopen and $E \subset U$ be such that $\Cp(E)=0$.
Assume that $u$ is a fine\/ \textup{(}super\/\textup{)}\-minimizer in $V$
for every finely open $V \pSubset U \setm E$.

Then $u$ is a fine\/ \textup{(}super\/\textup{)}minimizer in~$U$,
provided that 
one of the conditions~\ref{i-a}--\ref{i-c} in Theorem~\ref{thm-gen-seq-intro}
holds.
\end{prop}

These ``local-to-global principles''
play a crucial role when deducing the convergence
results in Section~\ref{sect-seq}, including Theorem~\ref{thm-gen-seq-intro}.
The obstacle problem, studied in~\cite{BBnonopen} and~\cite{BBLat4}, 
is also used as a fundamental tool.
As we shall see, the most difficult
part when showing that a function is a fine superminimizer
is often to show that it belongs to the fine local space $\Npploc(U)$.
A principal tool for achieving this 
is the following Caccioppoli-type inequality, 
proved in Section~\ref{sect-Cacc}.

\begin{thm}  \label{thm-cacc-intro}
  \textup{(Caccioppoli-type inequality)}
Let $u$ be a fine superminimizer in $U$ such that $|u-f|\le M$
  a.e.\ in $U$ for some $f\in\Np(U)$ and $M\ge0$.
Then
\begin{equation} \label{eq-cacc-lemma}
\int_{U} g_u^p \eta^p \,d\mu
\le 2^p \int_U \eta^p g_f^p \, d\mu + (4pM)^p \int_U g_\eta^p \, d\mu
\end{equation}
for every nonnegative $\eta \in \Np(X)$ with $\eta=0$ outside $U$.
\end{thm}

The compactly contained \p-strict subsets $V$, appearing in 
the definitions of $\Npploc$ and fine superminimizers,  as well as in 
Theorem~\ref{thm-gen-seq-intro} and 
Proposition~\ref{prop-local-global-intro}, were introduced by 
Kilpel\"ainen--Mal\'y~\cite{KiMa92} in $\R^n$ as a substitute for 
compactly contained open subsets, which are usually used to define 
the local space $\Nploc$ and supersolutions on open sets.
One of the
main difficulties, when dealing with the fine local space $\Npploc$
and fine superminimizers 
is that we do not know
the answer to the following question.

\begin{openprob} \label{openprob-pSubset-intro}
Given $A \pSubset U$, is there $W$ such that
$A \pSubset W \pSubset U$?
\end{openprob}

This is easily seen to be true if $U$ is open, but not known for
quasiopen $U$, even in $\R^n$.
As a consequence, some of our results and formulations are more cumbersome,
such as 
Theorem~\ref{thm-gen-seq-intro} and 
Proposition~\ref{prop-local-global-intro}
(and their generalizations Corollary~\ref{cor:localtoglobal-2} and Theorem~\ref{thm-gen-seq}).
This is also reflected in the formulation
of Theorem~4.2 in~\cite{KiMa92}.
Even the existence of $W$ such that
$A \psubset W\pSubset U$ would have important consequences,
cf.\ condition~\ref{m-d} in Corollary~\ref{cor:localtoglobal-2}
and Theorem~\ref{thm-gen-seq}.

The outline of the paper is as follows.
In Section~\ref{sect-prelim} we present the main background definitions
from nonlinear potential theory in metric spaces, while in
Section~\ref{sect-fine-cont} we introduce the fine topology
and the fine local Newtonian space $\Npploc(U)$.
Following our recent paper~\cite{BBLat4},
we introduce fine superminimizers and the fine obstacle problem
in Section~\ref{sect-fine-super-obst}.

In Section~\ref{sect-Cacc} we deduce the Caccioppoli-type inequality
(Theorem~\ref{thm-cacc-intro}),
whereas Section~\ref{sect-char}
is devoted to characterizations and local-to-global principles
for fine superminimizers.
Finally, in Section~\ref{sect-seq} we deduce several
convergence results for fine superminimizers
and one for fine minimizers.
Throughout the paper we utilize most of the results
obtained in \cite{BBLat4}.

\begin{ack}
A.~B. and J.~B. were supported by the Swedish Research Council,
grants 2016-03424 and 2020-04011 resp.\ 621-2014-3974 and 2018-04106.
Part of this research was done 
when V.~L. visited Link\"oping University in 2019 and 2022.
\end{ack}

\section{Notation and preliminaries}
\label{sect-prelim}

In this section, we introduce
the necessary metric space concepts used in this paper.
For brevity, we refer to
Bj\"orn--Bj\"orn--Latvala~\cite{BBLat1}, \cite{BBLat2}
for more
extensive introductions, and references to the literature.
See also the monographs Bj\"orn--Bj\"orn~\cite{BBbook} and
Heinonen--Koskela--Shanmugalingam--Tyson~\cite{HKST},
where the theory of upper gradients and Newtonian spaces
is thoroughly  developed with proofs.

Let $X$ be a metric space equipped
with a metric $d$ and a positive complete  Borel  measure $\mu$
such that $\mu(B)<\infty$ for all balls $B \subset X$.
We also assume that $1<p< \infty$.

A \emph{curve} is a continuous mapping from an interval,
and a \emph{rectifiable} curve is a curve with finite length.
We will only consider curves which are nonconstant, compact and rectifiable,
and they can therefore be parameterized by their arc length $ds$.
A property holds for \emph{\p-almost every curve}
if the curve family $\Ga$ for which it fails has zero \p-modulus,
i.e.\ there is $\rho\in L^p(X)$ such that
$\int_\ga \rho\,ds=\infty$ for every $\ga\in\Ga$.

  A measurable
  function $g:X \to [0,\infty]$ is a \p-weak \emph{upper gradient}
of $u:X \to \eR:=[-\infty,\infty]$
if for \p-almost all curves
$\gamma: [0,l_{\gamma}] \to X$,
\begin{equation*} 
        |u(\gamma(0)) - u(\gamma(l_{\gamma}))| \le \int_{\gamma} g\,ds,
\end{equation*}
where the left-hand side is $\infty$
whenever at least one of the
terms therein is infinite.
If $u$ has a \p-weak upper gradient in $\Lploc(X)$, then
it has a \emph{minimal \p-weak upper gradient}
$g_u \in \Lploc(X)$
in the sense that
$g_u \le g$ a.e.\
for every \p-weak upper gradient $g \in \Lploc(X)$ of $u$.

For measurable $u$, we let
\[
        \|u\|_{\Np(X)} = \biggl( \int_X |u|^p \, d\mu
                + \inf_g  \int_X g^p \, d\mu \biggr)^{1/p},
\]
where the infimum is taken over all \p-weak upper gradients of $u$.
The \emph{Newtonian space} on $X$ is
\[
        \Np (X) = \{u: \|u\|_{\Np(X)} <\infty \}.
\]

The space $\Np(X)/{\sim}$, where  $u \sim v$ if and only if $\|u-v\|_{\Np(X)}=0$,
is a Banach space and a lattice.
In this paper it is convenient to assume that functions in $\Np(X)$
 are defined everywhere (with values in $\eR$),
not just up to an equivalence class in the corresponding function space.

For an arbitrary set $A \subset X$,  we let
\[
  \Np_0(A)=\{u|_{A} : u \in \Np(X) \text{ and }
  u=0 \text{ on } X \setm A\}.
\]
Functions from $\Np_0(A)$ can be extended by zero in $X\setm A$ and we
will regard them in that sense when needed.

The  \emph{Sobolev capacity} of an arbitrary set $A\subset X$ is
\[
\CpX(A)=\inf_{u}\|u\|_{\Np(X)}^p,
\]
where the infimum is taken over all $u \in \Np(X)$ such that
$u\geq 1$ on $A$.

A property holds \emph{quasieverywhere} (q.e.)\
if the set of points  for which it fails has capacity zero.
The capacity is the correct gauge
for distinguishing between two Newtonian functions.
If $u \in \Np(X)$, then $v \sim u$ if and only if $v=u$ q.e.
Moreover, if $u,v \in \Np(X)$ and $v= u$ a.e., then $v=u$ q.e.
The superscript in $\CpX$ indicates dependence on the underlying space.
We will usually omit $X$ and write $\Cp$ instead of $\CpX$.

A set $U\subset X$ is \emph{quasiopen} if for every
$\varepsilon>0$ there is an open set $G\subset X$ such that $\Cp(G)<\varepsilon$
and $G\cup U$ is open. The quasiopen sets do not in general form a topology,
see Remark~9.1 in~\cite{BBnonopen}.
However it follows easily from the countable subadditivity of $\Cp$
that countable unions and finite intersections of quasiopen sets are
quasiopen.
Moreover, quasiopen sets are measurable.
Various characterizations of  quasiopen sets can be found in
Bj\"orn--Bj\"orn--Mal\'y~\cite{BBMaly}.
If $U\subset X$ is
quasiopen then $\CpU$ and $\CpX$ have the same zero sets in $U$,
  by Proposition~4.2 in~\cite{BBMaly}
  (and Remark~3.5 in Shan\-mu\-ga\-lin\-gam~\cite{Sh-harm}).

For a measurable set $E\subset X$, the Newtonian space $\Np(E)$
and the capacity $\CpE$ are
defined by
considering $(E,d|_E,\mu|_E)$ as a metric space in its own right.

If $E \subset A$ are subsets of $X$, then
the \emph{variational capacity} of $E$ with respect to $A$ is
\[
\cp(E,A) = \inf_{u}\int_{X} g_{u}^p\, d\mu,
\]
where the infimum is taken over all $u \in \Np_0(A)$
such that $u\geq 1$ on $E$.
If no such function $u$ exists then $\cp(E,A)=\infty$.

The measure  $\mu$  is \emph{doubling} if
there is $C>0$ such that for all balls
$B=B(x_0,r):=\{x\in X: d(x,x_0)<r\}$ in~$X$,
we have
$        0 < \mu(2B) \le C \mu(B) < \infty$,
where $\lambda B=B(x_0,\lambda r)$.
In this paper, all balls are open.

The space $X$ supports a \emph{\p-Poincar\'e inequality} if
there are $C>0$ and $\lambda \ge 1$
such that for all balls $B \subset X$,
all integrable functions $u$ on $X$ and all \p-weak upper gradients $g$ of $u$,
\begin{equation} \label{PI-ineq}
       \frac{1}{\mu(B)} \int_{B} |u-u_B| \,d\mu
       \le C \diam(B) \biggl( \frac{1}{\mu(\la B)}
              \int_{\lambda B} g^{p} \,d\mu \biggr)^{1/p},
\end{equation}
where $ u_B := \int_B u\, d\mu/\mu(B)$ and we implicitely assume that
$0<\mu(B)<\infty$ for all balls $B$.

In $\R^n$ equipped with a doubling measure $d\mu=w\,dx$, where
$dx$ denotes Lebesgue measure, the \p-Poincar\'e inequality~\eqref{PI-ineq}
is equivalent to the \emph{\p-admissibility} of the weight $w$ in the
sense of Heinonen--Kilpel\"ainen--Martio~\cite{HeKiMa}, see
Corollary~20.9 in~\cite{HeKiMa}
and Proposition~A.17 in~\cite{BBbook}.
Moreover, in this case $g_u=|\nabla u|$ a.e.\ if $u \in \Np(\R^n)$
and the capacities $\Cp$ and $\cp$ coincide
with the corresponding capacities in~\cite{HeKiMa}
(see \cite[Theorem~6.7]{BBbook} and \cite[Theorem~5.1]{BBvarcap}).

As usual, we set
$u_\limplus=\max\{u,0\}$.

\section{Fine topology, \texorpdfstring{$\Npploc(U)$}{Npfineloc(U)} and
\texorpdfstring{\p}{p}-strict subsets}
\label{sect-fine-cont}

\emph{Throughout the rest of the paper, we assume
  that $X=(X,d)$ is a complete metric space 
equipped with a doubling measure $\mu$
supporting a \p-Poincar\'e inequality
with  $1<p< \infty$.
Unless said otherwise, $U$ will be a nonempty
quasiopen set.}

\medskip

To avoid pathological situations we also assume that $X$
contains at least two points (and thus must be uncountable due to the
Poincar\'e inequality).
In this section we recall the basic facts about the fine topology
and Newtonian functions on quasiopen sets.

\begin{deff}\label{deff-thinness}
A set $E\subset X$ is  \emph{thin} at $x\in X$ if
\begin{equation*} 
\int_0^1\biggl(\frac{\cp(E\cap B(x,r),B(x,2r))}{\cp(B(x,r),B(x,2r))}\biggr)^{1/(p-1)}
     \frac{dr}{r}<\infty.
\end{equation*}
A set $V\subset X$ is \emph{finely open} if
$X\setminus V$ is thin at each point $x\in V$.
\end{deff}

In the definition of thinness,
we use the convention that the integrand
is 1 whenever $\cp(B(x,r),B(x,2r))=0$.
It is easy to see that the finely open sets give rise to a
topology, which is called the \emph{fine topology}.
Every open set is finely open, but the converse is not true in general.
A function $u : V \to \eR$, defined on a finely open set $V$, is
\emph{finely continuous} if it is continuous when $V$ is equipped with the
fine topology and $\eR$ with the usual topology.
Pointwise fine continuity is defined analogously.
The fine interior, fine boundary and fine closure of $E$
are denoted $\fineint E$, $\bdyp E$ and $\clEp$, respectively.
See \cite[Section~11.6]{BBbook} and
Bj\"orn--Bj\"orn--Latvala~\cite{BBLat1}
for further discussion on thinness and the fine topology
on metric spaces.

The following result shows that there is a close
connection between quasiopen and finely open sets.

\begin{thm} \label{thm-finelyopen-quasiopen}
  \textup{(Theorem~3.4 in~\cite{BBLat4})} 
The following conditions are equivalent for any set $U\subset X$\/\textup{:}
\begin{enumerate}
\item \label{tt-a}
$U$ is quasiopen\textup{;}
\item 
$U=V \cup E$ for some finely open $V$ and a set $E$
with $\Cp(E)=0$\textup{;}
\item \label{t-d}
$\Cp(U\setm \fineint U)=0$\textup{;}
\item \label{t-dd}
$U=\{x:u(x)>0\}$ for some $u\in \Np(X)$.
\end{enumerate}
\end{thm}

To define fine superminimizers, we first need
appropriate fine local Sobolev spaces.
Here \p-strict subsets will play a key role, as a substitute for relatively compact subsets.
Recall that $V\Subset U$ if $\overline{V}$ is a compact subset of $U$.

\begin{deff}\label{def-fineloc}
A set $A \subset U$  is a \emph{\p-strict subset} of $U$ if
there is a function $\eta \in \Np_0(U)$ such that $\eta =1$ on~$A$.
We will write $A \psubset U$, and similarly $A \pSubset U$
when also $A\Subset U$.

A function $u$ belongs to $\Npploc(U)$ if $u \in \Np(V)$
for all finely open $V \pSubset U$.
\end{deff}

In the definition of \p-strict subsets, it can be equivalently
required that in addition $0 \le \eta\le 1$,
as in Kilpel\"ainen--Mal\'y~\cite{KiMa92}; just
replace $\eta$ with $\min\{\eta,1\}_\limplus$.
Note that $A \psubset U$ if and only if $\cp(A,U)<\infty$.
Lemma~3.3 in Bj\"orn--Bj\"orn--Latvala~\cite{BBLat3}
shows that every finely open set
$V$ has a base of fine neighbourhoods $W \pSubset V$.
By Theorem~4.4 in~\cite{BBLat3},
functions in $\Npploc(U)$ are
finite q.e., finely continuous q.e.\ and quasicontinuous.

Throughout the paper, we consider minimal
\p-weak upper gradients in $U$.
For a function $u \in \Npploc(U)$ we say that
$g_{u,U}$ is a \emph{minimal \p-weak upper gradient} of $u$ in $U$
if 
\begin{equation*} 
  g_{u,U}=g_{u,V} \text{ a.e.\ in $V$} \quad
  \text{for every finely open $V \pSubset U$},
\end{equation*}
where $g_{u,V}$ is the
minimal \p-weak upper gradient of $u \in \Np(V)$  with respect to~$V$.
If $u \in \Np(U)$, then
this definition agrees
with the definition  of $g_{u,U}$
in
Section~\ref{sect-prelim}.
See \cite[Lemma~5.2 and Theorem~5.3]{BBLat3} for
the
existence, a.e.-uniqueness and minimality of $g_{u,U}$.
If $u\in \Nploc(X)$ then
the minimal \p-weak upper gradients $g_{u,U}$ and $g_u$ with respect to
$U$ and $X$, respectively, coincide a.e.\ in $U$,
see \cite[Corollary~3.7]{BBnonopen} or \cite[Lemma~4.3]{BBLat3}.
For this reason we drop $U$ from the notation and simply write $g_u$
from now on.

\begin{remark} \label{rmk-4.2}

If $u\in\Npploc(U)$ and $W \pSubset U$ is quasiopen,
then $W \setm \fineint W$ has
zero capacity by Theorem~\ref{thm-finelyopen-quasiopen}
(and is thus not seen by $g_u$ because of  \cite[Proposition~1.48]{BBbook}).
Hence $u \in \Np(W)=\Np(\fineint W)$.
Thus $u \in \Npploc(U)$
if and only if
$u \in \Np(W)$ for every quasiopen $W \pSubset U$.
\end{remark}

If follows from the definition of $\Npploc(U)$ and the
  properties of $\Np$ (see Section~\ref{sect-prelim}) that
functions in $\Npploc(U)$ are defined everywhere in $U$,
and that
if $u \in \Npploc(U)$ and $v=u$ q.e., then also $v \in \Npploc(U)$.
Moreover, if $u,v \in \Npploc(U)$ and $u=v$ a.e.\
  then $u=v$ q.e., as we shall now prove.
Due to the definition of $\Npploc(U)$ this
is not covered by
the results in \cite{BBbook}.

\begin{lem} \label{lem-ae-qe}
    Assume that $u,v \in \Npploc(U)$.
    Then the following are true\/\textup{:}
    \begin{enumerate}
 \item \label{eq-ae-qe-=}
If $u=v$ a.e.\ in $U$, then   $u=v$ q.e.\ in $U$.
 \item \label{eq-ae-qe->=}
   If $u \ge v$ a.e.\ in $U$, then   $u \ge v$ q.e.\ in $U$.
   \end{enumerate}
\end{lem}

\begin{proof}
Statement~\ref{eq-ae-qe-=} follows directly from \ref{eq-ae-qe->=},
so consider \ref{eq-ae-qe->=}.
By Lemma~3.3 in~\cite{BBLat3},
every $x \in \fineint U$ has a finely open neighbourhood
$W_x \pSubset U$.
By
Theorem~\ref{thm-finelyopen-quasiopen}
and
the quasi-Lindel\"of principle
\cite[Theorem~3.4]{BBLat3}
there is a countable subcollection $\{W_{j}\}_{j=1}^\infty$ of
$\{W_x\}_{x \in \fineint U}$
such that $\Cp(U \setm \bigcup_{j=1}^\infty W_{j})=0$.

By definition, $u, v\in \Np(W_{j})$,
and hence 
$u \ge v$ q.e.\ in $W_j$, $j=1,2,\dots$\,, by  Corollary~1.60 in \cite{BBbook},
and thus q.e.\ in $U$.
\end{proof}

\begin{remark}
  It follows from the covering obtained in the proof
  of Lemma~\ref{lem-ae-qe} that
  Corollaries~2.20 and~2.21 in~\cite{BBbook} extend to functions
  in $\Npploc(U)$.
  That is,
  if $u,v \in \Npploc(U)$ then
$g_u  = g_v$ a.e. on $\{x \in U : u(x)=v(x)\}$.
In particular,
\begin{alignat*}{2}
g_{\max\{u,v\}} &= g_{u} \chi_{\{u > v\}} + g_{v} \chi_{\{v \ge u\}}
&\quad & \text{a.e.},  \\
g_{\min\{u,v\}} &= g_{u} \chi_{\{v > u\}} + g_{v} \chi_{\{u \ge v\}}
&& \text{a.e.}
\end{alignat*}
\end{remark}

We will use the following simple lemma several times,
cf.\ Open problem~\ref{openprob-pSubset-intro}.

\begin{lem} \label{lem-seq-p-strict}
If $A \psubset U$, then there is a quasiopen  $W$ such that
$A \psubset W \psubset U$.
If $A$ is finely open, then $W$ can be chosen to be finely open.
\end{lem}

\begin{proof}
By assumption, there is $\eta \in \Np_0(U)$
such that $\eta=1$ on $A$.
We may also assume that $0\le \eta\le 1$.
Let $W:=\{x : \eta(x) > \tfrac12\}$, which is quasiopen by
Theorem~\ref{thm-finelyopen-quasiopen}.
The functions $(2\eta-1)_+\in\Np_0(W)$
and $\min\{2\eta,1\}\in\Np_0(U)$ then show that
$A \psubset W \psubset U$.

If $A$ is finely open, we can replace
$W$ by $\fineint W$, since
Theorem~\ref{thm-finelyopen-quasiopen}
implies that $\Np_0(W)=\Np_0(\fineint W)$.
\end{proof}

The following density result will play a crucial role.

\begin{prop} \label{prop-density}
\textup{(Proposition~4.5 in~\cite{BBLat4})}
Let $0\le \phi \in \Np_0(U)$.
Then there exist  finely open 
$V_j \pSubset U$ and bounded functions
$\phi_j \in \Np_0(V_j)$ such that
\begin{enumerate}
  \item
$V_j \subset V_{j+1}$ and $0 \le \phi_j \le \phi_{j+1} \le \phi$
    for $j=1,2,\dots$\,\textup{;}
\item
  $\|\phi-\phi_j\|_{\Np(X)} \to 0$ and $\phi_j(x) \to \phi(x)$
  for q.e.\ $x \in X$, as $j \to \infty$.
\end{enumerate}
\end{prop}

\section{Fine superminimizers and the obstacle problem}
\label{sect-fine-super-obst}

The following definition
is from \cite[Definition~5.1]{BBLat4}.
On unweighted $\R^n$, Kilpel\"ainen--Mal\'y~\cite[Section~3.1]{KiMa92}
gave an equivalent definition (of fine supersolutions) already in
1992, cf.\ Proposition~5.3 in \cite{BBLat4}.

\begin{deff}\label{def-finesuper}
A function $u \in \Npploc(U)$ is a
\emph{fine minimizer\/ \textup{(}resp.\ fine
superminimizer\/\textup{)}} in $U$ if
\[
\int_{V} g_{u}^p \, d\mu
\le  \int_{V} g_{u+\phi}^p \, d\mu
\]
for every finely open $V \pSubset U$
and for every (resp.\ every nonnegative)
$\phi \in \Np_0(V)$.
Moreover, $u$ is a \emph{fine subminimizer}
if $-u$ is a fine superminimizer.

A \emph{finely \p-harmonic function} is a finely continuous
fine minimizer.
\end{deff}

Note that, unlike for minimizers and \p-harmonic functions on
  open sets,
\emph{it is not known} whether every fine minimizer
can be modified on a set of zero capacity so that it becomes
finely continuous and thus finely \p-harmonic,
see the discussion in Section~9 in~\cite{BBLat4}.
Of course, being a function from $\Npploc(U)$,
every fine (super/sub)minimizer
is finely continuous at q.e.\ point.

By Lemma~5.4 in~\cite{BBLat4},
a function is a
fine minimizer if and only
if it is both a fine subminimizer and a fine superminimizer.
We refer to
\cite{BBLat4} for
further discussion on fine superminimizers
and the Newtonian space $\Npploc(U)$.
The following lattice property will be used
several times.

\begin{lem} \label{lem-min-supermin}
\textup{(Corollary~5.8 in~\cite{BBLat4})}
  If $u$ and $v$ are fine superminimizers in $U$,
then $\min\{u,v\}$ is also a fine superminimizer in $U$.
\end{lem}

The obstacle problem will be a fundamental tool when
studying fine (super)\-mi\-ni\-mi\-zers.
See~\cite{BBnonopen} and~\cite{BBLat4}
for earlier studies of the obstacle problem on nonopen sets
in metric spaces.
It was shown therein that such theories are
  not natural beyond
finely open or quasiopen sets.
In particular, it was shown in Bj\"orn--Bj\"orn~\cite[Theorem~7.3]{BBnonopen}
  that $\Np_0(E)=\Np_0(\fineint E)$ for an arbitrary set $E$.

\begin{deff} \label{deff-obst-E}
Assume that $U$ is a bounded nonempty
 quasiopen set with  $\Cp(X \setm U)>0$.
Let $f \in \Np(U)$ and $\psi : U \to \eR$.
Then we define
\begin{equation*}
    \K_{\psi,f}(U)=\{v \in \Np(U) : v-f \in \Np_0(U)
            \text{ and } v \ge \psi \ \text{q.e. in } U\}.
\end{equation*}
A function $u \in \K_{\psi,f}(U)$
is a \emph{solution of the $\K_{\psi,f}(U)$-obstacle problem}
if
\begin{equation*} 
       \int_U g^p_{u} \, d\mu
       \le \int_U g^p_{v} \, d\mu
       \quad \text{for all } v \in \K_{\psi,f}(U).
\end{equation*}
\end{deff}

Note that the boundary data $f$ are only required to belong to $\Np(U)$,
i.e.\ $f$ need not be defined on $\bdry U$
or the fine boundary $\bdyp U$.

If $\K_{\psi,f}(U) \ne \emptyset$,
then there is 
a solution $u$ of the $\K_{\psi,f}(U)$-obstacle problem,
which is unique q.e.,
see Theorem~4.2 in~\cite{BBnonopen}.
Moreover,
$u$ is a fine superminimizer,
by Theorem~6.2
in~\cite{BBLat4}.
A comparison principle for
obstacle problems was obtained in
Corollary~4.3 in~\cite{BBnonopen}.

\section{The Caccioppoli-type inequality}
\label{sect-Cacc}

Caccioppoli inequalities are important tools in
PDEs and nonlinear potential theory.
In this section, we are going to deduce the 
Caccioppoli-type inequality in Theorem~\ref{thm-cacc-intro}.
In Section~\ref{sect-seq}, we will use it
to deduce
the fundamental convergence theorems for increasing
and decreasing sequences of fine superminimizers.

\begin{proof}[Proof of Theorem~\ref{thm-cacc-intro}]
Assume first that $0 \le \eta \le 1$ and that $W:=\{x:\eta(x)>0\} \pSubset U$.
Then $W$ is quasiopen by  Theorem~\ref{thm-finelyopen-quasiopen},
and  $u\in\Np(W)$ by Remark~\ref{rmk-4.2}.
It follows from Lemma~\ref{lem-ae-qe} that
we can assume that $|u-f| \le M$ everywhere.
It thus follows
from the Leibniz and chain rules \cite[Theorem~2.15 and~2.16]{BBbook}
that $v:= u+\eta^p(f+M-u)\in\Np(W)$.
Lemma~2.4 in Kinnunen--Martio~\cite{KiMa02}
(or \cite[Lemma~2.18]{BBbook})
implies that
\[
g_v \le (1-\eta^p)g_u + \eta^p g_f + p(f+M-u)\eta^{p-1}g_\eta.
\]
(In \cite{KiMa02} and \cite{BBbook} it is assumed that $\eta \in \Lip(X)$,
  but since $u-f$ is bounded
  it is enough to require that $\eta \in \Np(X)$, $0 \le \eta \le 1$, as here.
  The proof is the same, just restrict attention to curves $\ga$ on
   which also $\eta$ is absolutely continuous.)
As $u$ is a fine superminimizer in $U$ and $0\le v-u \in \Np_0(W)$,
we obtain using the convexity of the function $t\mapsto t^p$ that
\begin{align*}
\int_{W} g_u^p \,d\mu &\le \int_{W} g_v^p \,d\mu
  \le \int_{W} \biggl( (1-\eta^p)g_u +
     \eta^{p} \Bigl( g_f + p|f+M-u| \frac{g_\eta}{\eta} \Bigr) \biggr)^p \,d\mu \\
&\le \int_{W} (1-\eta^p) g_u^p \,d\mu
          + 2^p \int_{W} \eta^p \biggl( g_f^p  + (2pM)^p\frac{g^p_\eta}{\eta^p} \biggr) \,d\mu.
\end{align*}
Since $g_u\in L^p(W)$, subtracting the first integral on the right-hand
side from both sides of the inequality yields \eqref{eq-cacc-lemma}
and proves the statement when $0 \le \eta \le 1$ and $\{x:\eta(x)>0\} \pSubset U$.

For general $0\le\eta \in \Np_0(U)$, Proposition~\ref{prop-density}
provides us with finely open $V_j \pSubset U$
and an increasing sequence of nonnegative bounded functions
$\eta_j \in \Np_0(V_j)$ such that
$\|\eta-\eta_j\|_{\Np(X)} \to 0$ 
and $\eta_j \nearrow \eta$ a.e., as $j \to \infty$.
Applying \eqref{eq-cacc-lemma} to $V_j$ and suitable multiples of $\eta_j$,
we get
\[ 
  \int_{U} g_u^p \eta^p \,d\mu
  = \lim_{j \to \infty} \int_{V_j} g_u^p \eta_j^p \,d\mu
  \le 2^p \int_{U} \eta^p g_f^p \, d\mu + (4pM)^p \int_{U} g_\eta^p \, d\mu
\] 
by using monotone convergence.
\end{proof}

In the next sections we will several times need to
deduce that certain functions belong to $\Npploc(U)$.
In order to do so, the following result will be convenient.

\begin{lem}
\label{lem:bdd-super-seq-uj-W'}
Let $E$ be a set with $\Cp(E)=0$ and
assume that $u=\lim_{j \to \infty} u_j$ q.e.\ in $U$, where each
$u_j$, $j=1,2,\dots$\,, is a fine superminimizer
in $V$ for 
every finely open $V \pSubset U\setm E$.
Let $W \psubset U$ be a bounded measurable set.
Assume that
$|u_j-f| \le M$
a.e.\ in $U$, $j=1,2,\dots$\,, for some  $f\in\Np(U)$ and $M\ge0$.
Then $u \in \Np(W)$ and
\begin{equation}   \label{eq-liminf-g-uj-W'}
\int_W g_u^p\,d\mu \le \liminf_{j\to\infty} \int_W g_{u_j}^p\,d\mu
\le 2^p \int_{U} g_f^p \,d\mu + (4pM)^p \int_{U}g_{\eta}^p \,d\mu,
\end{equation}
where $\eta\in\Np_0(U)$ is any function such that $0\le\eta\le1$ in
$U$ and $\eta=1$ in $W$.
\end{lem}

A natural question is of course whether
$u$ is a fine superminimizer
in $U$, but we postpone that to later.
As we shall see, this proposition is useful also for the constant
sequence $u_j=u$, $j=1,2,\dots$\,, as it gives the 
Caccioppoli-type inequality~\eqref{eq-cacc-lemma} under weaker assumptions 
than Theorem~\ref{thm-cacc-intro}.

\begin{proof}
Let $\eta \in \Np_0(U)$ be such that $0\le\eta\le1$ in $U$
  and $\eta=1$ in $W$.
Since $\Cp(E)=0$, we have $\eta \in \Np_0(U\setm E)$.
By Proposition~\ref{prop-density},
there are
finely open sets $V_k \pSubset U\setm E$ and
nonnegative bounded functions $\eta_k \in \Np_0(V_k)$ such that
$\|\eta-\eta_k\|_{\Np(X)} \to 0$
and $\eta_k \nearrow \eta$ q.e.\ as $k \to \infty$.

Since each $u_j$ is a fine superminimizer
in every $V_k$, we get using
the Caccioppoli-type inequality (Theorem~\ref{thm-cacc-intro})
and monotone convergence that
\begin{align*}
\int_{W} g_{u_j}^p \,d\mu
& \le \int_{U} g_{u_j}^p \eta^p \,d\mu
= \lim_{k \to \infty} \int_{V_k} g_{u_j}^p \eta_k^p \,d\mu  \\
& \le \lim_{k \to \infty} \int_{V_k} (2^p\eta^p_k g_f^p + (4pM)^p g_{\eta_k}^p) \,d\mu
\le \int_{U} (2^p g_f^p + (4pM)^p g_{\eta}^p) \,d\mu < \infty,
\end{align*}
i.e.\ $\{u_j\}_{j=1}^\infty$ is a bounded sequence in $\Np(W)$
(since it is bounded in $L^p(W)$ by assumption).
Hence, by Corollary~6.3 in~\cite{BBbook}, $u\in\Np(W)$ and
\eqref{eq-liminf-g-uj-W'} holds.
\end{proof}

\section{Characterizations of fine superminimizers}
\label{sect-char}

In this section we will deduce several characterizations of fine
superminimizers, which will
be used later on.
The first
characterization, based on Lemma~\ref{lem-min-supermin},
will be used in the proofs of
Proposition~\ref{prop-char-truncation-new}
and  Corollary~\ref{cor-incr-lim-supermin}.

\begin{prop} \label{prop-char-truncation}
Let $v$ be a fine superminimizer in $U$ and $u\in\Npploc(U)$.
Then $u$ is a fine superminimizer in $U$ if and only if
$\min\{u,v+k\}$ is a fine superminimizer in $U$
for every $k=1,2,\dots$\,.
\end{prop}

\begin{proof}
Assume first that $u_k:=\min\{u,v+k\}$ is a fine superminimizer
for every $k=1,2,\dots$\,.
Let $V \pSubset U$ be finely open
and $\phi \in \Np_0(V)$ be nonnegative.
As $u,v \in \Npploc(U)$, they are finite q.e.\ and $u_k\nearrow u$
q.e.\ in $V$, as $k \to \infty$.
Let $E_k:=\{x\in V: u(x)>v(x)+k\}$.
Since $V$ is bounded, we have $\mu(E_k)\to0$.
Using monotone and dominated convergence (since
  $v+\phi\in\Np(V)$), we then get
\[   
\int_V g_{u_k+\phi}^p \, d\mu
= \int_{V\setm E_k} g_{u+\phi}^p \, d\mu + \int_{E_k} g_{v+\phi}^p \, d\mu
\to  \int_V g_{u+\phi}^p \, d\mu, \quad \text{as $k\to\infty$.}
\]   
Applying this also with $\phi\equiv0$
and using that $u_k$ is a fine superminimizer yields
\[
\int_V g_{u}^p \, d\mu = \lim_{k\to\infty} \int_V g_{u_k}^p \, d\mu
\le \lim_{k\to\infty} \int_V g_{u_k+\phi}^p \, d\mu = \int_V g_{u+\phi}^p \, d\mu.
\]
As $V$ and $\phi$
were arbitrary this shows that $u$ is
a fine superminimizer.
The converse implication follows directly from
Lemma~\ref{lem-min-supermin}.
\end{proof}

The following
result
will be used to prove
Theorem~\ref{thm-gen-seq-intro}
and its more general version Theorem~\ref{thm-gen-seq}.

\begin{prop} \label{prop-char-truncation-new}
Assume that $v$ is a fine superminimizer in $U$ and let
$u:U\to \eR$ be such that $|u-f|\le M$ a.e.\ in $U$ for some
$f\in\Np(U)$ and $M\ge0$.

If $\min\{u,v+k\}$ is a fine superminimizer in $U$
for every $k=1,2,\dots$\,, then $u\in\Np(W)$ for every 
bounded quasiopen $W\psubset U$.
Moreover,
$u$ is a fine superminimizer in $U$.
\end{prop}

\begin{proof}
Let $W \psubset U$ be a bounded quasiopen set
and $\eta\in\Np_0(U)$ be such that
$0\le\eta\le1$ and $\eta=1$ in $W$.
By Proposition~\ref{prop-density},
there are finely open sets
$V'_1\subset V'_2\subset \dots \pSubset U$
and functions $0\le\eta'_k \in \Np_0(V'_k)$ such that
$\|\eta-\eta'_k\|_{\Np(X)} \to 0$
and $\eta'_k \nearrow \eta$ q.e.\ as $k \to \infty$.

For each $k=1,2,\dots$\,, let
\begin{align*}
M_k &> \|g_{v-f}\|_{L^p(V'_k)}, \\
V_k &= \{x\in W: f(x)< v(x)+k\}, \\
W'_k &= \{x\in V'_k: f(x) <v(x)+k+M_k\}.
\end{align*}
Note that both $V_k$ and $W'_k$
are quasiopen, by 
Theorem~\ref{thm-finelyopen-quasiopen}.
Since $f, v\in\Np(V'_k)$ and $\eta'_k\in\Np_0(V'_k)$, we see that
\[
0\le \eta_k :=
\min \biggl\{ \frac{(v+k+M_k-f)_\limplus}{M_k}, 2\eta'_k,1 \biggr\} \in\Np_0(V'_k),
\]
by Lemma~2.37 in~\cite{BBbook}.
The \p-weak upper gradient of $\eta_k$ is estimated by
\begin{align}
\int_{V'_k} g_{\eta_k}^p \, d\mu
&\le \int_{V'_k} \max \Bigl\{ \frac{g_{v-f}}{M_k}, 2g_{\eta'_k}\Bigr\}^p \, d\mu
\label{eq-est-g-eta-k}   \\
&\le \int_{V'_k}  \biggl( \frac{g_{v-f}^p}{M_k^p} + 2^p g_{\eta'_k}^p \biggr)\, d\mu
< 1 + 2^p \int_{U} g_{\eta'_k}^p \, d\mu
\to 1 + 2^p \int_{U} g_{\eta}^p \, d\mu,     \nonumber
\end{align}
as $k\to\infty$.
Moreover, $\eta_k=0$ in $V'_k\setm W'_k$ and hence
$\eta_k\in\Np_0(W'_k)$.
On the other hand, $\eta_k=1$ on $W_k := \{ x\in V_k: \eta'_k(x)>\tfrac12\}$,
and so
\[
W_k \psubset W'_k \pSubset U.
\]
In $W'_k\pSubset U$  we have by Lemma~\ref{lem-ae-qe} that
\[
u = \min\{u,v+k+M_k+M\}
\quad \text{q.e.},
\]
and so $u$ is a fine superminimizer in $W'_k$
(since we may assume that $M_k$ and $M$ are integers).
The Caccioppoli inequality (Theorem~\ref{thm-cacc-intro} with $U$
replaced by $W'_k$) now implies that $u\in\Np(W_k)$ and
\begin{align}   \label{eq-int-Wk-f-eta-k}
\int_{W_k} g_{u}^p \,d\mu
& \le 2^p \int_{W'_k} \eta_k^p g_f^p \, d\mu + (4pM)^p \int_{W'_k}g_{\eta_k}^p \, d\mu
\nonumber \\
&
\le 2^p \int_{U} g_f^p \, d\mu + (4pM)^p \int_{V'_k} g_{\eta_k}^p \, d\mu.
\end{align}
Now, for q.e.\  $x\in W$, there exists $k$ such that $\eta'_k(x)>\tfrac12$ and
$f(x)<v(x)+k$ and hence $x\in W_k$.
It follows that $W=\bigcup_{k=1}^\infty W_k \cup Z$, where $\Cp(Z)=0$.

Since $W_1\subset W_2\subset \dots$\,, inserting
\eqref{eq-est-g-eta-k} into \eqref{eq-int-Wk-f-eta-k} together with
monotone convergence implies that $\int_{W} g_{u}^p \,d\mu<\infty$.
As $|u-f|\le M$ a.e.\ in $W$, we conclude that $u\in\Np(W)$.

Applying this to an arbitrary finely open $W \pSubset U$
shows that $u\in\Npploc(U)$,
and the superminimizing property follows from
  Proposition~\ref{prop-char-truncation}.
\end{proof}

The following lemma relates fine superminimizers to obstacle problems.

\begin{lem}\label{lem-super-obst}
  Let $u$ be a fine superminimizer in $U$,
  and let $W \subset U$ be a bounded quasiopen set
  such that $u \in \Np(W)$ and $\Cp(X \setm W)>0$.
  Then $u$ is a solution of the $\K_{u,u}(W)$-obstacle problem.
\end{lem}

\begin{proof}
Let $v \in \K_{u,u}(W)$.  
Then $\phi :=v-u\ge 0$ q.e.\ in $W$ and $\phi\in\Np_0(W)$.
By Lemma~5.5 in~\cite{BBLat4},
\[
  \int_{W} g_u^p \, d\mu
  =   \int_{\{\phi  = 0\}} g_u^p \, d\mu + \int_{\{\phi  > 0\}} g_u^p \, d\mu
  \le    \int_{\{\phi  = 0\}} g_v^p \, d\mu + \int_{\{\phi  > 0\}} g_v^p \, d\mu
  =     \int_{W} g_v^p \, d\mu.
\]
Thus $u$ is a solution of the $\K_{u,u}(W)$-obstacle problem.
\end{proof}

We next turn to ``local-to-global principles''
for fine superminimizers, which will be crucial later on.

\begin{prop}\label{prop-superobstacle}
Let $u \in \Npploc(U)$ and $E$ be a set with $\Cp(E)=0$.
Then the following are equivalent\/\textup{:}
\begin{enumerate}
\item \label{t-a}
  $u$ is a fine superminimizer in $U$\textup{;}
\item \label{t-b}
  $u$ is a fine superminimizer in $V$
  for every finely open
\textup{(}or equivalently quasiopen\/\textup{)}
$V \pSubset U\setm E$
such that $\Cp(X \setm V)>0$\textup{;}
\item\label{t-c}
$u$ is a solution of the $\K_{u,u}(V)$-obstacle problem
for every finely open\/
\textup{(}or equivalently quasiopen\/\textup{)}
$V \pSubset U\setm E$ such that $\Cp(X \setm V)>0$.
\end{enumerate}
\end{prop}

Note that the condition $\Cp(X \setm V)>0$
  is automatically satisfied for arbitrary  $V \Subset U\setm E$
  if either $X$ is unbounded or $U \setm E \ne X$,
as in these cases $\overline{V}$ is a compact subset of $X$
whose complement is a nonempty open set.

\begin{proof}[Proof of Proposition~\ref{prop-superobstacle}]
The implication \ref{t-a}\imp\ref{t-b} is trivial,
while \ref{t-b}\imp\ref{t-c} follows from Lemma~\ref{lem-super-obst}.
It remains to show  \ref{t-c}\imp\ref{t-a}.

To this end, consider first the case when $X$ is unbounded or
$U \setm E \ne X$.
Let $V \pSubset U$ be  finely open and $\phi\in\Np_0(V)$ be nonnegative.
Since $\Cp(E)=0$, we have $\phi\in\Np_0(V\setm E)$.
Proposition~\ref{prop-density}
provides us with finely open $V_j \pSubset V\setm E$
and an increasing sequence of bounded functions
$0\le\phi_j \in \Np_0(V_j)$ such that $\phi_j \nearrow \phi$ q.e.\
and
$\|\phi-\phi_j\|_{\Np(X)} \to 0$ as $j \to \infty$.
Then $X\setm \overline{V}_j$
is open and nonempty and hence $\Cp(X \setm V_j)>0$.
Thus, by assumption, $u$ is a solution of the $\K_{u,u}(V_j)$-obstacle problem.
Since $u+\phi_j\in\K_{u,u}(V_j)$,
we get for each $j$ that
\[
\biggl(\int_{V_j} g_{u}^p \, d\mu \biggr)^{1/p}
\le \biggl( \int_{V_j} g_{u+\phi_j}^p \, d\mu \biggr)^{1/p}
\le \biggl( \int_{V} g_{u+\phi}^p \, d\mu \biggr)^{1/p} + \|\phi-\phi_j\|_{\Np(X)}.
\]
Letting $j\to\infty$ and using monotone convergence shows that
\[
\int_{V} g_{u}^p \, d\mu \le  \int_{V} g_{u+\phi}^p \, d\mu.
\]
As $V$ and $\phi$ were arbitrary, it
follows that $u$ is a fine superminimizer in $U$.

Finally, assume that $U\setm E=X$ is bounded.
By assumption, $u$ is a solution of the $\K_{u,u}(\Om)$-obstacle problem
  for every open $\Om \Subset X$ with $\Cp(X \setm \Om)>0$.
  Hence, it follows from Proposition~9.25 in~\cite{BBbook} that
  $u$ is a superminimizer in $X$, and thus
a fine superminimizer in $U=X$, by Corollary~5.6 in~\cite{BBLat4}.
\end{proof}

We are now ready to prove Proposition~\ref{prop-local-global-intro},
as well as the following more general version of it.

\begin{cor}  \label{cor:localtoglobal-2}
Let $E$ be a set with $\Cp(E)=0$ and
assume that $u$ is a fine\/ \textup{(}super\/\textup{)}minimizer in $V$
for every finely open $V \pSubset U \setm E$
and that one of the following conditions holds\/\textup{:}
\begin{enumerate}
\item \label{m-a}  \label{m-first}
$u$ is a.e.-bounded.
\item \label{m-b}
$u\in\Npploc(U)$.
\item \label{m-c}
For every finely open $V' \psubset U$ there are
$f\in\Np(V')$ and $M\ge0$ {\rm(}depending on $V'$\/{\rm)}
such that $|u-f| \le M$ a.e.\ in $V'$.
\item \label{m-d}  \label{m-last}
For every finely open $W \pSubset U$
there is a quasiopen 
$W'$, $f\in\Np(W')$
and $M\ge0$ {\rm(}depending on $W'$\/{\rm)} such that $W \psubset W' \subset U$ and
$|u-f|\le M$ a.e.\ in~$W'$.
\end{enumerate}
Then $u$ is a fine\/ \textup{(}super\/\textup{)}minimizer in~$U$.
\end{cor}

For $E=\emptyset$ and open $U$,
an analog of Corollary~\ref{cor:localtoglobal-2} is immediate,
  without the additional assumptions \ref{m-first}--\ref{m-last},
  see Proposition~9.21 in~\cite{BBbook}.
The fine case is more involved, 
and we do not know if it is true
without assuming one of \ref{m-first}--\ref{m-last},
cf.\ the discussion around  Open problem~\ref{openprob-pSubset-intro}.
Letting $U=B(0,1)$, $E=\{0\}$ and $u(x)=|x|^{(p-n)/(p-1)}$
in unweighted $\R^n$, $1<p<n$,
shows that Corollary~\ref{cor:localtoglobal-2}
would fail if \ref{m-first}--\ref{m-last} were omitted, even when $U$ is open.

\begin{proof}
The implication \ref{m-a}\imp\ref{m-d} is clear,
while \ref{m-c}\imp\ref{m-d} holds by Lemma~\ref{lem-seq-p-strict}.
When \ref{m-d} holds and $W \pSubset U$ is finely open,
it follows from
Lemma~\ref{lem:bdd-super-seq-uj-W'} (with $u_j=u$ and $U$ replaced by $W'$)
that $u \in \Np(W)$, and thus  $u \in \Npploc(U)$.
Hence  \ref{m-b} holds in all cases.
The conclusion now follows directly from
    Proposition~\ref{prop-superobstacle}.
\end{proof}

\begin{remark} \label{rmk-local-to-global-sheaf}
Combining 
Theorem~4.2\,(a)  in Kilpel\"ainen--Mal\'y~\cite{KiMa92} with 
Corollary~\ref{cor:localtoglobal-2}
shows that 
a function on a quasiopen set in
unweighted $\R^n$,  which  either  belongs to 
$\Npploc$  or is bounded,
is a fine supersolution if and only if it is a quasilocal fine supersolution
in the sense of \cite[Definition~4.1]{KiMa92}.
We do not know if the corresponding equivalence holds in metric spaces.
If it does then it would 
answer Open problem~9.23 in~\cite{BBbook} in the affirmative 
and at least partially yield
the sheaf property,
cf.\ the discussion 
after Theorem~\ref{thm-gen-seq-intro}.

The definition of quasilocal fine supersolutions
is also based on and yields the sheaf property.
Moreover, Example~5.2 in~\cite{KiMa92} demonstrates that, in contrast to fine supersolutions, 
unbounded quasilocal fine supersolutions on open sets
can fail to be supersolutions in the usual sense.
See 
the discussions after Theorem~\ref{thm-gen-seq-intro} above
and before Example~5.2 in~\cite{KiMa92},
and also Lemma~5.7 in Latvala~\cite{LatPhD}.

\end{remark}

\section{Sequences of fine superminimizers}
\label{sect-seq}

In this section we are going to deduce convergence results for
fine superminimizers.
Kilpel\"ainen--Mal\'y~\cite[Theorem~4.3]{KiMa92} deduced 
monotone convergence results for 
quasilocal fine supersolutions on unweighted $\R^n$.
When combined with Theorem~4.2\,(a) (with $V=U$) in~\cite{KiMa92}
their results
yield Proposition~\ref{prop-decr-lim-supermin-Np}
and Corollary~\ref{cor-incr-lim-supermin} in $\R^n$
under the additional assumptions that 
$U$ is bounded and $u \in \Np(U)$.
For the counterpart of Proposition~\ref{prop-decr-lim-supermin-Np} they also require that
$u$ is bounded from below.
Using our local-to-global Proposition~\ref{prop-superobstacle},
one can now 
remove the first additional assumption of \cite{KiMa92} and replace
the second one with $u \in \Npploc(U)$ or $u$ being bounded.

\begin{prop}        \label{prop-decr-lim-supermin-Np}
Let  $\{u_j\}_{j=1}^\infty$ be a decreasing sequence of
fine superminimizers in~$U$.
If $u:=\lim_{j\to\infty} u_{j}\in\Npploc(U)$, then $u$ is a fine superminimizer in $U$.
\end{prop}

\begin{proof}
Let $V \pSubset U$ be a finely open set such that $\Cp(X \setm V)>0$,
and $v$ be a solution of the $\K_{u,u}(V)$-obstacle problem.
As $u_j$ is a solution of the $\K_{u_j,u_j}(V)$-obstacle problem
(by Lemma~\ref{lem-super-obst}),
the comparison
principle \cite[Corollary~4.3]{BBnonopen}
implies that $v \le u_j$ q.e.\ in $V$.
Since this holds for all $j$, we have $v \le u$ q.e.\ in $V$.

On the other hand, by the definition of the obstacle problem,
$v \ge u$ q.e.\ in $V$, and thus $u=v$ q.e.\ in $V$.
Hence, $u$ is also a solution of the $\K_{u,u}(V)$-obstacle problem.
By Proposition~\ref{prop-superobstacle}, $u$ is a fine superminimizer in $U$.
\end{proof}

\begin{thm}  \label{thm-incr-lim-supermin-new}
Let  $\{u_j\}_{j=1}^\infty$ be an increasing sequence of
fine superminimizers in $U$ such that $|u_j -f|\le M$
a.e.\ in $U$ for some $f\in\Np(U)$ and some $M\ge0$.

Then $u:=\lim_{j\to\infty} u_j$
is a fine superminimizer in $U$.
\end{thm}

The proof has been inspired by the proof of
Theorem~6.1 in Kinnunen--Martio~\cite{KiMa02}.

\begin{proof}
Lemma~\ref{lem:bdd-super-seq-uj-W'}
implies that $u\in \Npploc(U)$ and that \eqref{eq-liminf-g-uj-W'}
holds for every finely open $W\pSubset U$.
To prove the superminimizing property,
let $V_0\pSubset U$ be  finely open
and let $\phi_0 \in \Np_0(V_0)$ be nonnegative.
Let $\eps>0$ be arbitrary.
By Proposition~\ref{prop-density} there is
a finely open
$V' \pSubset V_0$ and a bounded $0\le\phi\in\Np_0(V')$ so that
\begin{equation}   \label{eq-choose-V-eps-1}
\|\phi-\phi_0\|_{\Np(V_0)} <\eps.
\end{equation}
By Lemma~3.3 in~\cite{BBLat3},
$V_0$ has a base of fine neighbourhoods $W' \pSubset V_0$,
and thus, by the quasi-Lindel\"of principle
\cite[Theorem~3.4]{BBLat3}, we can find a countable collection
$W_j \pSubset V_0$ of finely open sets such that 
$\Cp(V_0 \setm \bigcup_{j=1}^\infty W_j)=0$.
Since $g_u\in L^p(V_0)$, dominated convergence implies that
there is $m$ such that
\begin{equation}   \label{eq-choose-V-eps-2}
\int_{V_0\setm V} g_u^p\,d\mu < \eps,
\quad \text{where } V=V' \cup \bigcup_{j=1}^m W_j.
\end{equation}
Moreover, $V \pSubset V_0$ is finely open.
By the definition of \p-strict subsets, there exists
$\eta_0\in\Np_0(V_0)$ such that $0\le\eta_0\le 1$ and $\eta_0=1$ in $V$.
Then
\[
W:=\bigl\{x\in V_0: \eta_0(x)>\tfrac12\bigr\}\pSubset U
\]
and $W$ is quasiopen, by Theorem~\ref{thm-finelyopen-quasiopen}.
Moreover, $\eta:=(2\eta_0-1)_\limplus \in\Np_0(W)$
and $\eta=1$ on $V$.

Let $v=u+\phi$. 
Then $0\le v-u_j\le \phi+2M$ a.e.
Since $\eta$ and $\phi$ are bounded, it thus follows
from the Leibniz rule \cite[Theorem~2.15]{BBbook} 
that $\psi_j:=\eta(v-u_j)\in\Np_0(W)$.
As in the proof of 
the Caccioppoli-type inequality (Theorem~\ref{thm-cacc-intro})
in Section~\ref{sect-Cacc},
we have
\[
g_{u_j+\psi_j} \le (1-\eta)g_{u_j} + \eta g_v + (v-u_j)g_\eta
	\quad \text{a.e. in } W,
\]
by Lemma~2.4 in Kinnunen--Martio~\cite{KiMa02}
(or \cite[Lemma~2.18]{BBbook}).
Since $W\pSubset U$ is quasiopen, we get from the
superminimizing property of $u_j$ that
\begin{align*}
\biggl(\int_{W} g_{u_j}^p \, d\mu \biggr)^{1/p}
   &\le \biggl(\int_{W} g_{u_j+\psi_j}^p \, d\mu \biggr)^{1/p} \\
   &\le \biggl(\int_{W} ((1-\eta) g_{u_j} + \eta g_v)^p \, d\mu \biggr)^{1/p}
      + \biggl(\int_{W} (v-u_j)^p g_\eta^p \, d\mu \biggr)^{1/p} \\
   &=: \alp_j + \be_j.
\end{align*}
Using the elementary inequality
\[
      (\alp + \be)^p \le \alp^p + p \be (\alp + \be)^{p-1},
	\quad \alp,\be \ge0,
\]
together with the convexity of $t \mapsto t^p$, we obtain that
\begin{align*}
\int_{W} g_{u_j}^p \, d\mu
    & \le \int_{W} ((1-\eta) g_{u_j} + \eta g_v)^p \, d\mu
        + p \be_j (\alp_j + \be_j)^{p-1} \\
    & \le \int_{W} (1-\eta) g_{u_j}^p \, d\mu
	+ \int_{W} \eta g_v^p \, d\mu + p \be_j (\alp_j + \be_j)^{p-1}.
\end{align*}
As $u_j \in \Np(W)$, we can subtract the first term on the right-hand
side from both sides of the inequality and obtain that
\begin{equation} \label{eq-ono}
\int_{V} g_{u_j}^p \, d\mu
     \le \int_{W} \eta g_{u_j}^p \, d\mu
         \le \int_{W} g_v^p \, d\mu + p\be_j (\alp_j + \be_j)^{p-1}.
\end{equation}
Since $2\eta_0>1$ on $W$, the
Caccioppoli-type inequality (Theorem~\ref{thm-cacc-intro})
shows that
\begin{align*}
\al_j &\le \biggl( \int_{W} g_{u_j}^p\,d\mu \biggr)^{1/p}
        + \biggl( \int_{W}g_v^p\,d\mu \biggr)^{1/p} \\
&\le 
   \biggl( \int_{V_0} (4^p\eta_0^p g_{f}^p
     + (8pM)^p g_{\eta_0}^p) \,d\mu \biggr)^{1/p}
            + \biggl( \int_{W} g_v^p\,d\mu \biggr)^{1/p},
\end{align*}
i.e.\ the sequence $\{\alp_j\}_{j=1}^\infty$ is bounded.
At the same time, since $g_\eta=0$ a.e.\ in~$V$ and $v=u$ outside~$V$,
we have by dominated convergence and the fact that
$|(u-u_j)g_\eta|\le 2M g_\eta \in L^p(W)$,
\[
\be_j = \biggl(\int_{W \setm V} (u-u_j)^p g_\eta^p \, d\mu \biggr)^{1/p}
	\to 0,
	\quad \text{as } j \to \infty.
\]
Altogether this shows that the last term in \eqref{eq-ono}
tends to $0$, as $j \to \infty$.
Using \eqref{eq-ono} together with \eqref{eq-liminf-g-uj-W'} applied to $V$,
we obtain that
\[
\biggl(\int_{V} g_{u}^p \, d\mu \biggr)^{1/p}
    \le \biggl(\int_{V_0} g_v^p \, d\mu \biggr)^{1/p}
   < \biggl(\int_{V_0} g_{u+\phi_0}^p \, d\mu \biggr)^{1/p} + \eps,
\]
where the last estimate follows from 
\eqref{eq-choose-V-eps-1} and the fact that
\[
g_v = g_{u+\phi} \le g_{u+\phi_0} + g_{\phi-\phi_0} \quad \text{a.e.\ in } V_0.
\]
Finally, \eqref{eq-choose-V-eps-2} yields
\[
\int_{V_0} g_{u}^p \, d\mu
    = \int_{V} g_{u}^p \, d\mu + \int_{V_0\setm V} g_{u}^p \, d\mu
   <      \biggl( \biggl(\int_{V_0} g_{u+\phi_0}^p \, d\mu \biggr)^{1/p}
        + \eps \biggr)^p + \eps.
\]
As $\eps>0$ was arbitrary this completes the proof.
\end{proof}

\begin{cor}   \label{cor-incr-lim-supermin}
Let  $\{u_j\}_{j=1}^\infty$ be an increasing sequence of
fine superminimizers in~$U$.
If $u:=\lim_{j\to\infty} u_j\in\Npploc(U)$ then it is a fine superminimizer in $U$.
\end{cor}

\begin{proof}
In view of Proposition~\ref{prop-superobstacle},
it suffices to show that $u$ is a fine superminimizer in
  every finely open $V\pSubset U$.

For each such $V$ and $M\ge0$, the functions $\min\{u_j,u_1+M\}$, $j=1,2,\dots$
(which are fine superminimizers by
Lemma~\ref{lem-min-supermin}),
satisfy the assumptions of Theorem~\ref{thm-incr-lim-supermin-new} with $U$ and $f$
replaced by $V$ and $u_1$.
Hence, $\min\{u,u_1+M\}$ is a fine superminimizer in $V$.
Proposition~\ref{prop-char-truncation} shows that $u$ is a fine
superminimizer in $V$.
Since $V$ was arbitrary, the claim follows.
\end{proof}

We are now ready to prove Theorem~\ref{thm-gen-seq-intro}
as well as the following more general version of it.

\begin{thm} \label{thm-gen-seq}
Let $\{u_j\}_{j=1}^\infty$ be a sequence of fine superminimizers
in~$U$ such that for every finely open $W \pSubset U$
there are $f_0\in\Np(W)$ and $m \ge 1$
{\rm(}depending on $W$\/{\rm)} such that $u_j \ge f_0$
a.e.\ in $W$, $j=m,m+1,\dots$\,.
Assume that
one of the
following conditions from
  Corollary~\ref{cor:localtoglobal-2}
holds for $u:=\liminf_{j\to\infty} u_j$\/\textup{:}
\begin{enumerate}
\item \label{gen-a}  \label{gen-first}
$u$ is a.e.-bounded.
\item \label{gen-b}
$u \in \Npploc(U)$.
\item \label{gen-c}
For every finely open $V' \psubset U$ there are
$f\in\Np(V')$ and $M\ge0$ {\rm(}depending on $V'$\/{\rm)}
such that $|u-f| \le M$ a.e.\ in $V'$.
\item \label{gen-d} \label{gen-last}
For every finely open $W \pSubset U$
  there is a quasiopen $W'$, $f\in\Np(W')$
  and $M\ge0$ {\rm(}depending on $W'$\/{\rm)} such that $W \psubset W' \subset U$ and
$|u-f|\le M$ a.e.\ in~$W'$.
\end{enumerate}
Then $u$ is a fine superminimizer in $U$.
\end{thm}

\begin{proof} 
Consider first the case when $X$ is bounded and $U=X$.
Since $X$ is open, each $u_j$
is a  standard superminimizer in $X$, by
Corollary~5.6 in~\cite{BBLat4}.
It then follows from \cite[Proposition~9.4 and Corollary~9.14]{BBbook}
that $u_j$ is q.e.-constant, and thus so is $u$,
since any of the assumptions \ref{gen-first}--\ref{gen-last}
excludes the theoretical possibility that $u=\infty$ q.e.
Hence $u$ is a fine superminimizer. So in the rest of the proof
we assume that $U \ne X$ or that $U=X$ is unbounded.

Consider $l\ge0$ and a finely open set $W\pSubset U$.
Let $f_0\in \Np(W)$ and $m$
be as in the assumptions.
Note that $X \setm \overline{W}$ is a nonempty open set and
thus $\Cp(X \setm W)>0$.
Let $v$ be a solution of the $\K_{f_0,f_0}(W)$-obstacle problem.
As each $u_j$ is a solution of the $\K_{u_j,u_j}(W)$-obstacle problem
(by Lemma~\ref{lem-super-obst}),
the comparison
principle \cite[Corollary~4.3]{BBnonopen}
implies that $v \le u_j$ q.e.\ in $W$ for all $j\ge m$.
By Theorem~6.2 in~\cite{BBLat4},
$v$ is a fine superminimizer in $W$.
For every fixed integer $k\ge m$, the functions
\[
v_{k,j} = \min\{u_k,\dots,u_j, v+l\}, \quad j\ge k,
\]
are fine superminimizers in $W$, by Lemma~\ref{lem-min-supermin}.
Let $V \pSubset W$ be finely open.
Applying Lemma~\ref{lem:bdd-super-seq-uj-W'} on $W$
with $M=l$ and $f=v\in\Np(W)$ (and $U$ resp.\ $W$ replaced  by $W$ resp.\ $V$)
to the decreasing sequence
$v_{k,j}$ shows that $v_k:=\lim_{j\to\infty} v_{k,j}\in\Np(V)$.
As $V$ was arbitrary, we see that $v_k \in\Npploc(W)$.

Proposition~\ref{prop-decr-lim-supermin-Np}, applied to $W$, then implies that $v_k$ is
a fine superminimizer in $W$.
Note that $v\le v_k\nearrow \min\{u, v+l\}$ q.e.\ in $W$ and that $v_k$
satisfy the assumptions of Theorem~\ref{thm-incr-lim-supermin-new}
with $U$, $f$ and $M$ replaced by $W$, $v$ and $l$.
It follows that $\min\{u,v+l\}$ is a fine superminimizer in $W$
for every $l\ge0$.
Since any of the assumptions \ref{gen-first}--\ref{gen-last}
(with $f \equiv 0$ in \ref{gen-first} and $f=u$ in \ref{gen-b})
implies that $|u-f| \le M$ a.e.\ in $W$,
Proposition~\ref{prop-char-truncation-new} (applied with $U$ replaced
by $W$) implies that $u$ is a fine superminimizer in $W$.

Since $W\pSubset U$ was arbitrary,
Corollary~\ref{cor:localtoglobal-2} concludes the proof.
\end{proof}

For sequences of fine minimizers, the following result is a direct
consequence of Theorem~\ref{thm-gen-seq}, applied to both
$\{u_j\}_{j=1}^\infty$ and $\{-u_j\}_{j=1}^\infty$.
Corollary~\ref{cor-monotone-min-intro} is 
a special case of this result.

\begin{cor} 
Let  $\{u_j\}_{j=1}^\infty$ be a sequence of
fine minimizers in $U$ such that for every finely open $W \pSubset U$
there are $f_0, f_1\in\Np(W)$ and $m \ge 1$
{\rm(}depending on $W$\/{\rm)} such that $f_0 \le u_j \le f_1$
a.e.\ in $W$, $j=m,m+1,\dots$\,.
Assume that $u_j\to u$ q.e.\ in $U$ and that
one of the
conditions~\ref{gen-a}--\ref{gen-d} in Theorem~\ref{thm-gen-seq} holds.
Then $u$ is a fine minimizer in $U$.
\end{cor}


\begin{thebibliography}{99}

\bibitem{BBbook} \book{Bj\"orn, A. \AND Bj\"orn, J.}
        {\it Nonlinear Potential Theory on Metric Spaces}
    {EMS Tracts in Mathematics {\bf 17},
        European Math. Soc., Z\"urich, 2011}

\bibitem{BBvarcap} \art{\auth{Bj\"orn}{A} \AND \auth{Bj\"orn}{J}}	
        {The variational capacity with respect to nonopen sets in metric spaces}
	{Potential Anal.} {40} {2014} {57--80}

\bibitem{BBnonopen} \art{\auth{Bj\"orn}{A} \AND \auth{Bj\"orn}{J}}	
	{Obstacle and Dirichlet problems on arbitrary nonopen sets
          in metric spaces, and fine topology}
        {Rev. Mat. Iberoam.} {31} {2015} {161--214}

\bibitem{BBLat1} \art{\auth{Bj\"orn}{A}, \auth{Bj\"orn}{J}  \AND
    \auth{Latvala}{V}}
        {The weak Cartan property for the \p-fine topology on metric spaces}
	{Indiana Univ. Math. J.} {64} {2015} {915--941}

\bibitem{BBLat3} \art{\auth{Bj\"orn}{A}, \auth{Bj\"orn}{J}  \AND
    \auth{Latvala}{V}}
        {Sobolev spaces,  fine gradients and quasicontinuity on quasiopen sets}
	{Ann. Acad. Sci. Fenn. Math.} {41} {2016} {551--560}

\bibitem{BBLat2} \art{\auth{Bj\"orn}{A}, \auth{Bj\"orn}{J}  \AND
    \auth{Latvala}{V}}
        {The Cartan, Choquet and Kellogg properties of the
        fine topology on metric spaces}
	{J. Anal. Math.} {135} {2018} {59--83}

\bibitem{BBLat4} \arttoappear{\auth{Bj\"orn}{A}, \auth{Bj\"orn}{J}  \AND
    \auth{Latvala}{V}}
        {The Dirichlet problem for \p-minimizers on finely open sets in metric spaces}
        {Potential Anal.}

\bibitem{BBMaly} \art{\auth{Bj\"orn}{A}, \auth{Bj\"orn}{J}  \AND
    \auth{Mal\'y}{J}}
        {Quasiopen and \p-path open sets, and characterizations of quasicontinuity}
	{Potential Anal.} {46} {2017} {181--199}

\bibitem{buttazzo-dalMaso} \art{\auth{Buttazzo}{G} \AND \auth{Dal Maso}{G}}
     {An existence result for a class of shape optimization problems}
     {Arch. Ration. Mech. Anal.} {122} {1993} {183--195}

\bibitem{Fug} \book{\auth{Fuglede}{B}}
         {Finely Harmonic Functions}
         {Springer, Berlin--New York, 1972}

\bibitem{Fug74} \art{Fuglede, B.}
        {Fonctions harmoniques et fonctions finement harmoniques}
        {Ann. Inst. Fourier\/ \textup{(}Grenoble\/\textup{)}}
        {24{\rm:4}}{1974}{77--91}

\bibitem{FuscoMZ} \art{\auth{Fusco}{N}, \auth{Mukherjee}{S}
        \AND \auth{Zhang}{Y. R.-Y}}
      {A variational characterisation of the second eigenvalue of the
        \p-Laplacian on quasi open sets}
      {Proc. Lond. Math. Soc.} {119} {2019} {579--612}

\bibitem{HeKiMa} \book{\auth{Heinonen}{J},
	\auth{Kilpel\"ainen}{T}
	\AND \auth{Martio}{O}}
        {Nonlinear Potential Theory of Degenerate Elliptic Equations}
        {2nd ed., Dover, Mineola, NY, 2006}

\bibitem{HKST} \book{\auth{Heinonen}{J}, \auth{Koskela}{P},
	\auth{Shanmugalingam}{N} \AND \auth{Tyson}{J. T}}
       {Sobolev Spaces on Metric Measure Spaces}
	{New Mathematical Monographs {\bf 27}, Cambridge Univ. Press,
        Cambridge, 2015}

\bibitem{KiMa92} \art{Kilpel\"ainen, T. \AND Mal\'y, J.}
        {Supersolutions to degenerate elliptic equation on quasi open sets}
        {Comm. Partial Differential Equations}
        {17} {1992} {371--405}

\bibitem{KiMa02} \art{Kinnunen, J. \AND Martio, O.}
         {Nonlinear potential theory on metric spaces}
         {Illinois Math. J.} {46} {2002} {857--883}

\bibitem{LuMaZa} \book{\auth{Luke\v{s}}{J}, \auth{Mal\'y}{J} \AND
         \auth{Zaj\'i\v{c}ek}{L}}
         {Fine Topology Methods in Real Analysis and Potential Theory}
         {Springer, Berlin--Heidelberg, 1986}

\bibitem{LatPhD} \book{\auth{Latvala}{V}}
        {Finely Superharmonic Functions of Degenerate Elliptic Equations}
        {Ann. Acad. Sci. Fenn. Ser. A I Math. Dissertationes {\bf 96}
        {(1994)}}

\bibitem{Lat00} \art{Latvala, V.}
        {A theorem on fine connectedness}
        {Potential Anal.} {12} {2000} {221--232}

\bibitem{Sh-harm} \art{Shanmugalingam, N.}
         {Harmonic functions on metric spaces}
         {Illinois J. Math.}{45}{2001}{1021--1050}

\end{thebibliography}
\end{document}